\documentclass{article}

\usepackage{lastpage}

\usepackage{fancyhdr}

\usepackage{graphicx}
\usepackage{mathtools}
\usepackage{enumerate}
\usepackage{amssymb}
\usepackage{mathptmx}     
\usepackage{csquotes}

\usepackage{latexsym}
\newtheorem{theorem}{Theorem}
\newtheorem{proposition}{Proposition}

\newtheorem{definition}{Definition}
\newtheorem{proof}{Proof}
\newtheorem{example}{Example}
\newtheorem{remark}{Remark}

\RequirePackage{fix-cm}

\def\.#1{\hfill#1\kern.5em\vrule\kern-.5em}

\newtheorem{consequence}{Consequence}

\begin{document}

\title{Bolzano and the Part-Whole Principle}

\author{Kate\v{r}ina Trlifajov\'{a}}



\maketitle

\begin{abstract}
The embracing of actual infinity in mathematics leads naturally to the question of comparing the sizes of infinite collections. The basic dilemma is that the Cantor Principle (CP), according to which two sets have the same size if there is a one-to-one correspondence between their elements, and the Part-Whole Principle (PW), according to which the whole is greater than its part, are inconsistent for infinite collections. Contemporary axiomatic set-theoretic systems, for instance ZFC, are based on CP. PW is not valid for infinite sets. 

In the last two decades the topic of sizes of infinite sets has resurfaced again in a number of papers. A question of whether it is possible to compare the sizes to comply with PW has been risen and researched. 

Bernard Bolzano in his 1848 Paradoxes of the Infinite dealt with principles of introducing infinity into mathematics. He created a theory of infinite quantities that respects PW and which is based on sums of infinite series. We extend Bolzano's theory and create a constructive method for determining the set size of countable sets so that the cardinality of finite sets is preserved and PW is valid. Our concept is close to the numerosity theory from the beginning of this century but it is simpler and more intuitive. In the results, we partly agree with the theory of numerosities. The sizes of countable sets are uniquely determined, but they are not linearly ordered.

\end{abstract}

\section{Prologue}
The core of the problem is usually demonstrated using Galileo's paradox. Its original formulation is, let us try to considerate, somewhat misleading.
Salviati compares the collection of all natural numbers and the collection of all squares, i.e. the second powers of natural numbers. 
\medskip

$$1, \quad 2, \quad 3, \quad \ 4, \quad \ 5, \ \dots$$
$$1, \quad 4, \quad 9, \quad 16, \quad 25,\ \dots $$
Salviati explains to Simplicio that on one hand all numbers, including both squares and non-squares, are more than the squares alone. On the other hand, there are as many squares as the corresponding number of roots, since every square has its own root and every root its own square and all the numbers are roots. 

The question which already partly contains the answer remains unpronounced. 
How is the collection of squares formed? There are two options. Either we select them from numbers and then there is certainly less of them. 

Or we create them from natural numbers. We use Galileo's assumption which seems, but does not have to be, self-evident:

\begin{quote} Every number is a root of some square. (Galilei 1638/1914, p. 40).\end{quote} 

It means that the squares of natural numbers are also natural numbers. Thus we must also take their squares, and so on. We collect all the squares into a whole. This collection has as many elements as the collection of all natural numbers. 

Acceptance of both solutions simultaneously leads to a contradiction called a paradox. So the question is, what is the meaning of the three dots that signify the continuation of an infinite sequence. 

$$\ast \ast \ast $$

The existence of the infinite set of natural numbers is guaranteed by the \emph{axiom of infinity}\footnote{The \emph{axiom of infinity} is a definitive positive answer to the more than two millennia old question about the existence of actual infinity. It says that there is an infinite set. $$(\exists x)(\emptyset \in x \wedge ((\forall y)(y \in x \Rightarrow y \cup \{y\} \in x)$$} in Zermelo-Fraenkel's axiomatic set theory (ZF). The set of squares used to be described 
\footnote{This type of description is used by William Parker, for instance: \bf Even $ = \{2n; n \in \mathbb N\}$, Odd $= \{2n+1; n \in \mathbb N\}$, \rm (Parker 2013, p. 601)}
$$S = \{n^2; n \in \mathbb N\}.$$
This is an inaccurate description of a set whose existence is guaranteed by the \emph{axiom of specification}.\footnote{The \emph{axiom of specification} states the construction of a subset defined by a special property. Formally, let $x$ be a set and $\varphi$ a formula of the language of ZF. Then there is a set $z$
$$z = \{y; y \in x \wedge \varphi(y)\}$$

} We write more precisely
$$S = \{n \in \mathbb N; (\exists m)(m \in \mathbb N \wedge n = m^2\}$$
This is obviously a subset of natural numbers, even a proper subset, since most natural numbers are not squares. At the same time, there is a a one-to-one correspondence of the set of natural numbers $\mathbb N$ on the set of squares $S$ and thus from the perspective of ZF the both sets have the same number of elements.

\section{Bolzano's Conception}

The largest part of Bolzano's conception of infinity can be found in \emph{Paradoxes of the Infinite} \emph{[Paradoxien des Unedlichen]} (PU) (Bolzano 1851/2004) written in 1848 and in \emph{Theory of Science [Wissenschaftslehre]} (WL) (Bolzano 1837/2014) published in 1837.\footnote{We refer to these books by their abbreviations followed by a number of the paragraph.} Cantor's fundamental work had been presented much later in a series of papers published between 1872 and 1883, especially in the last one which he entitled \emph{Foundations of the General Theory of Manifolds} and in \emph{Contributions to the Founding of the General Theory of Transfinite Numbers} (Cantor 1895/1915) 
from 1893 and 1895.

\subsection{Sets, collections, multitudes, series}\label{collection}

Both Bolzano and Cantor were pioneers and advocates of actual infinity in mathematics and sought a way to approach it consistently. Their initial idea is similar. In order to embrace actual infinity, it is necessary to collect it into a whole. Cantor's basic concept is a \emph{set} , Bolzano's basic concept is a \emph{collection}. Although both terms express the wholes, there is a slight but substantial difference between them. 
Cantor's definition of a set is: 
\begin{quote} By a \emph{set [Menge]} we are to understand any collection $M$ into a whole of definite and separate objects $m$ of our intuition or our thought. These objects are called \emph{elements} of $M$. In signs we express this thus
$$M = \{m\}$$ (Cantor, 1895/1915, p. 93)
\end{quote}
The set, as defined by Cantor, and as it is understood today, is a whole that contains all its elements. It is a new object which is at a different \enquote{higher} level than its elements. This is also evident from the notation, which is still used, where we use curly brackets, for example
$\{0,1\}$ is the set containing $0$ and $1$. 

\centerline{$\ast \ast \ast$}

Bolzano's basic notion is a \emph{collection}. 
\begin{quote} The concept which underlies the conjunction \enquote{and}\dots which can be expressed most suitably by the words: 
a \emph{collection [Inbegriff]} of certain things or a whole consisting of certain parts. (PU \S 3)
\end{quote}

Bolzano's \emph{collection} is also a whole containing certain things, but the way of their grasp is different. The things are connected by the conjunction \enquote{and}. Consequently, the collection is at the \enquote{same} level as the things it contains. That's also why Bolzano calls them \emph{parts}. 

A special type of a collection is a multitude.

\begin{quote} A collection which we put under a concept so that the arrangement of its parts is unimportant I call a \emph{multitude [Menge]}. (PU \S 4). \end{quote}

Bolzano's multitude and Cantor's set share many features (Rusnock 2012, p. 157). 
Both Bolzano and Cantor used the same German word \emph{Menge}. It is possible that Cantor adopted it after reading Bolzano's \emph{Paradoxes of the Infinite} (Simons 2005, p. 143). Nevertheless, Bolzano scholars agree that Bolzano's multitudes cannot be interpreted as Cantor's sets. They cannot be interpreted mereologically, too. \begin{quote} Attempts to interpret Bolzano's collections as sets or mereological wholes in the modern sense face insuperable difficulties and I think Peter Simon was exactly right in deeming Bolzano's theory \emph{sui generis}. (Simons 1998, p. 87), (Rusnock 2012, p. 155). \end{quote}
Peter Simon has proposed the English term \emph{multitude} for Bolzano's \emph{Menge} (Simons 1998, p. 95). This translation has been widely accepted. (Russ 2004), (Rusnock 2013).
\medskip

A special type of a multitude is \emph{a plurality [Vielheit]}. Its parts are all considered as units of a certain kind $A$. (PU \S 4). A plurality or a multitude is finite if we can designate things in it by some finite number $n$. (PU \S 22).  And a it is \emph{infinite} if it is greater than every finite plurality.  
\begin{quote}
I shall call a plurality which is greater than every finite one, i.e. a plurality which has the property that every finite multitude represents only a part of it, an \emph{infinite plurality}. (PU \S 9).
\end{quote}

Bolzano's series are also collections, but unlike multitudes, their terms are ordered according to some rule.

\begin{quote}
If a given collection of things $A,B,C, \dots , M, N, \dots$ has the property that for every part $M$, some one and only one, other part $N$ can be identified of a kind we can \emph{determine} by the \emph{same rule} for all parts of the collection either $N$ by its relationship to $M$ or $M$ by its relationship to $N$, then I call this collection a \emph{series [Reihe]} and its parts the \emph{terms} of the series. (PU \S 7).
\end{quote}

The series can have a finite number of terms. Then their sums are finite pluralities. But they can also have an infinite number of terms.

\begin{quote}
There is sometimes a greater and sometimes a smaller multitude of terms in the series being discussed. In particular there can be so many of them that this sequence, to the extent that it is to exhaust \emph{all} these units, may have absolutely \emph{no last term}. (PU \S 9). 
\end{quote}
And those are the series Bolzano is interested in. They represent infinite pluralities.

\subsection{Part-whole principle}\label{PW}
In his \emph{Foundations of a General Theory of Manifolds} from 1883 Cantor chose the existence of a one-to-one correspondence as the only correct criterion for the equality of two sets, the \emph{Cantor principle} (CP). As a result the \emph{Part-Whole Principle} (PW) is not valid for infinite sets. 
Cardinal numbers that denote the number of elements of sets
\begin{quote} \dots are the result of double act of abstraction, abstraction of the nature of its various elements and of the order in which they are given, (Cantor 1895, p. 86). \end{quote} 

Note that they are also an abstraction from the \emph{way sets are formed}. As a consequence, the cardinal number $\aleph_0$, which denotes the number of elements of the set of all natural numbers, is the same for all its infinite subsets, for example the set of squares. The set of all rational numbers and of all algebraic numbers also have the same cardinal number $\aleph_0$. 

From Cantor's famous diagonal proof follows that the set of real numbers has a greater cardinal number. The set of all points on any line segment, the set of all points in the whole infinite three-dimensional or n-dimensional infinite space and all other continua have the same cardinal number. 

$$\ast \ast \ast$$

Bolzano was aware of the existence of a one-to-one correspondence between some infinite sets, such that one is a proper subset of the other. He calls this property a \emph{highly remarkable peculiarity} (PU, \S 21). However, he insists on the validity of PW. A mere one-to-one correspondence is not a sufficient condition to determine that two sets have the same number of elements. 
\begin{quote} Merely from this circumstance we can - as we see - in no way conclude \emph{that these multitudes are equal to one another if they are infinite} with respect to the plurality of their parts (i.e. if we disregard all differences between them). But rather they are able to have a relationship of unequality in their plurality, so that one of them can be presented as a whole, of which the other is a part. (PU \S 21).
\end{quote}
Only in some cases it is possible to determine that two sets have the same number of elements.
\begin{quote} An equality of these multiplicities can only be concluded if some other reason is added, such as that both multitudes have exactly the same \emph{determining ground} [Bestimmungsr\H{u}nde], e.g. they have exactly the same \emph{way of being formed} [Entstehungsweisse]. (PU \S 21). \end{quote} 
Bolzano does not explain a meaning of the term \enquote{determining ground} nor the \enquote{way of being formed} in \emph{Paradoxes of the Infinite}. But from other Bolzano's texts follows that to \enquote{determine an object} means to describe all representation that the object falls under. The determination is complete if the representation of an object is unique (\v{S}ebest\'{i}k 1992, p. 460). However, the exact mathematical meaning is not entirely clear. 

Some examples of equal multitudes are given in PU. Bolzano denotes $Mult(b - a)$ the multitude of all quantities lying between two boundary quantities $a$ and $b$.\footnote{Bolzano apparently has in mind the open interval $(a,b)$ of his \emph{measurable numbers} which are isomorphic to our \emph{real numbers}, see (Russ, Trlifajov\'{a} 2016).}
\begin{quote} There will be enumerable equations of the following form:
$$Mult(8-7) = Mult(13-12)$$
and also of the form
$$Mult(b-a) : Mult(d-c) = (b-a):(d-c)$$ 
against the correctness of which no valid objection can be made (PU \S 29). 
\end{quote}
The sufficient condition for the equality and even for the ratio of two intervals is an isometry relation, i.e. distance-preserving one-to-one correspondence. The second example is geometrical. 
\begin{quote}
Every spatial extension that is not only similar to another but also geometrically equal (i.e. coincides with it in all characteristics that are conceptually representable through comparison with a given distance) must also have an equal multitude of points. (PU \S 49. 2.) 
\end{quote}

\subsection{Infinite quantities}

Bolzano in his \emph{Paradoxes of the Infinite} deals with infinite quantities that arise as infinite series. I will briefly introduce his theory using his own examples, as the principles on which he built are obvious from those examples. He presents the series of 
natural numbers and the series of squares. 
$$P = 1 + 2 + 3 + 4 + \dots + \text{in inf.} $$
$$S = 1 + 4 + 9 + 16 + \dots + \text{in inf.} $$

It is no coincidence that Bolzano marks the end of an infinite series with the symbol \enquote{ ... in inf.} An important assumption is that all the infinite series have one and the same multitude of terms, unless explicitly stated otherwise. About the two listed series $P, S$ Bolzano writes

\begin{quote}{The multitude of terms in both series is certainly the same. By raising every single term of the series $P$ to the square into the series $S$, we alter merely the nature (magnitude) of these terms, not their plurality.} \end{quote} 
Based on this assumption and comparing the corresponding terms of the series, it turns out that the second series $S$ is greater than the first series $P$.
$S$ is even infinitely greater than $P$, meaning that $S$ is greater than every finite multiple of $P$. If we subtract\footnote{Let us notice that Bolzano subtracts the series pointwise.} successively the series $P$ from $S$ then we obtain the differences
\begin{itemize}
\item $S - P = 0 + 2 + 6 + 12 + \dots + \text{in inf.} $
\item $S - 2 P = -1 + 0 + 3 + 8 + \dots + \text{in inf.} $
\item $S - m P = (1-m) + \dots + (n^2-mn) + \dots + \text{in inf.} $
\end{itemize}
In these series, only a finite multitude of terms, namely $m-1$, is negative, and the $m$-th term is 0, but all successive terms are positive and grow indefinitely. Every series is positive. 

The same assumption follows from the following Bolzano's example. If we remove several elements from an infinite series, it has less elements by the exact amount we removed. Its sum is smaller exactly by the sum of the elements removed. For example, a series
$$N_0 = 1 + 1 + 1 + 1 + \dots + \text{in inf.} $$
contains only 1. It differs from the series 
$$N_n = \underbrace{\dots }_n 1 + 1 + 1 + \dots + \text{in inf.}, $$ which looks similar but we create it by adding $1$ only from the $(n+1)$-st term. The first $n$ terms are omitted. 
\begin{quote}We obtain the certain and quite unobjectionable equation 
$$n = N_0 - N_n.$$
from which we see that two infinite quantities $N_0$ and $N_n$ have a completely definite finite difference $n$. (PU \S 29). \end{quote}

\subsection{Interpretation}

Can the Bolzano series of natural numbers be consistently interpreted in contemporary mathematics? Yes, surprisingly yes, using sequences of partial sums (Trlifajov\'{a} 2018). Sequences have the advantage of having a clearly defined number of terms, and each term has a specific order.\footnote{It is not by chance that the ultraproduct which is the basis of the non-standard analysis is defined on \emph{sequences} of real numbers.} By converting the series to a sequence, we satisfy the requirement that the series always have one and the same number of elements. We can express naturally some specific Bolzano's quantities that are otherwise difficult to describe, such as $N_n$ and its difference from $N_0$. 

\begin{definition} Let $a_1 + a_2 + a_3 + \dots + \text{in inf.}$ be the Bolzano series of natural numbers, $a_i \in \mathbb N.$ It corresponds to the sequence $(s_1, s_2, s_3, \dots ) \in \mathbb N^\mathbb N$ where for all $n \in \mathbb N, s_n = a_1 + \dots + a_n.$
$$a_1 + a_2 + a_3 + \dots + \text{in inf.} \sim (s_1, s_2, s_3, \dots ) = (s_n),$$ 
\end{definition}

\begin{example} 

\begin{itemize}
\item $P = 1 + 2 + 3 + 4 + \dots + \text{in inf.} \sim (1, 3, 6, 10, \dots ) = (\frac {n \cdot (n+1)}{2})$
\item $S = 1 + 4 + 9 + 16 + \dots + \text{in inf.} \sim (1, 5, 16, 32, \dots ) = (\frac{n(n+1)(2n+1)}{6})$
\item $N_0 = 1 + 1 + 1 + 1 + \dots + \text{in inf.} \sim (1, 2, 3, \dots) = (n)$
\item $N_n = \underbrace{\dots }_n 1 + 1 + 1 + + \dots + \text{in inf.} \sim ( \underbrace{0, \dots, 0}_n, 1, 2, 3, \dots )$

\end{itemize}
\end{example}
We interpret Bolzano's series of natural numbers as non-decreasing sequences of numbers. The following definition corresponds to Bolzano's way of adding series. 
\begin{definition}\label{plus} Let $(a_n), (b_n)$ be two sequences of natural numbers. We define their sum and their product componentwise:
$$(a_n) + (b_n) = (a_n + b_n).$$
$$(a_n) \cdot (b_n) = (a_n \cdot b_n).$$
If $(\forall n)(a_n > b_n)$ then we define the difference of two sequences 
$$(a_n) - (b_n) = (a_n - b_n).$$ 

\end{definition}

Bolzano demanded associativity and commutativity for terms of the series. It means that if we change the order of the finite amount of terms of the series, the sum of the series does not change. Therefore, we define that two sequences are equal, if their terms are equal starting from a sufficiently large index. Similarly, we define the ordering. 

\begin{definition}\label{equality} Let $(a_n), (b_n)$ be two sequences of natural numbers. 
$$(a_n) =_\mathcal F (b_n) \text{ if and only if } (\exists m)(\forall n)(n > m \Rightarrow a_n = b_n).$$
$$(a_n) <_\mathcal F (b_n) \text{ if and only if }(\exists m)(\forall n)(n > m \Rightarrow a_n < b_n).$$
\end{definition}
\begin{proposition} The relation $=_\mathcal F$ is an equivalence. The result of this equality is the same as if we define equality by the Fr\'{e}chet filter on natural numbers. The relation $<_\mathcal F$ is a strict partial ordering. 
\end{proposition}
\begin{proof} The Fr\'{e}chet filter is the set of all complements of finite subsets of natural numbers. 
$$\mathcal F = \{A \subseteq \mathbb N; \mathbb N \setminus A \textrm{ is finite}\}$$
Let $(a_n), (b_n)$ be two sequences of natural numbers. Their equality by the Fr\'{e}chet filter is defined
$$(a_n) =_\mathcal F (b_n) \text{ if and only if } \{n; a_n = b_n\} \in \mathcal F.$$
If two sequences are equal from a sufficiently large index then they differ only in finitely many terms. Thus they are equal by the Fr\'{e}chet filter. Vice versa if two sequences are equal by the Fr\'{e}chet filter then they must be equal from a sufficiently large index. So, the designation $<_\mathcal F$ is justified. 

A strict partial ordering must be irreflexive and transitive. These properties are obvious from Definition \ref{equality}.
\end{proof}
\begin{theorem}
Let $S = \{(a_n), a_n \in \mathbb N \wedge (\forall n)(a_n \leq a_{n+1}\}$ be the set of non-decreasing sequences of natural numbers. Then the structure $(S, + , \cdot, =_\mathcal F, <_\mathcal F)$ where the equality and the ordering is defined by Fr\'{e}chet filter is a partial ordered non-Archimedean commutative semiring. 
\end{theorem}

\begin{proof}
The properties of a commutative semiring are associativity, commutativity, distributivity and the existence of neutral element both for addition and multiplication. They are clearly valid from Definition \ref{plus}. The sequence $(n^2)$ is infinitely greater than the sequence $(n)$ and it is infinitely greater than any constant sequence, thus it is a non-Archimedean structure. 

\end{proof}

However, Bolzano goes further. Terms of his infinite series may not only be natural numbers but also other quantities, especially Bolzano's measurable numbers, which are isomorphic to contemporary real numbers. 
We interpret them as sequences of real numbers, arithmetic operations and order are defined similarly as before. Then we can prove the following theorem. (Russ, Trlifajov\'{a} 2016). 
\begin{theorem}
Let $\mathbb R^\mathbb N$ be the set of sequences sequences of real numbers. Then the structure $(\mathbb R^\mathbb N, + , \cdot, =_\mathcal F, <_\mathcal F)$ where the equality and the ordering is defined by Fr\'{e}chet filter is a partial ordered non-Archimedean commutative ring. 
\end{theorem}

One can introduce the \emph{cheap version of non-standard analysis} on this structure. (Tao 2012). It is constructive but less powerful than the full version of non-standard analysis.

Only in modern mathematics of the second half of the 20th century did Abraham Robinson show that if we use an ultrafilter instead of Fr\'{e}chet filter in the definition of equality and an ordering of sequences we obtain an ultraproduct of hyperreal numbers - a linearly ordered non-Archimedean field which shares \emph{all} first order properties with the field $\mathbb R$ due to the \emph{transfer principle}.\footnote{Jan Berg was right when he entitled his book on Bolzano \emph{Ontology without Ultrafilters.} (Berg 1992).}

\section{Sizes of calculable sets}

\subsection{Calculable sets}

Bolzano does not determine systematically the relationship between the multitudes of natural numbers and their subsets with some exceptions.\footnote{In the \emph{Theory of Science} from 1837 Bolzano demonstrates that there are infinitely many concepts which encompass infinitely many objects and there are infinite differences in extensions (breadth) of these concepts. He presents the concept of natural numbers, denoted by $n$, which is infinitely greater than the concept of their squares, $n^2$, it is is infinitely greater than the concept of the fourth powers, $n^4$, and so on. 
\begin{quote}
If we designate any whole number whatever by the letter $n$ as an
abbreviation, then the numbers $n, n^2, n^4, n^8, n^{16}, n^{32}$, express concepts,
each of which undoubtedly encompasses infinitely many objects (namely
infinitely many numbers.) \dots 
The extension of the first one is infinitely greater than that of the second one, it is infinitely greater that the third one and so on. 
Now since the sequence
$n, n^2, n^4, n^8, n^{16}, n^{32}$ can be extended as far as we please, we have in it
an example of an infinite sequence of concepts, each of which is an infinite
number of times broader that its successor. (WL \S 102). \end{quote} 
However Jan Berg (1964, p. 177) found in a letter to R. Zimmermann on March 9, 1848 a passage in which Bolzano questions this view. 
\begin{quote} The thing has not only become unclear
but, as I have just come to realize, completely false. If one designates
by $n$ the concept of an arbitrary whole number, then it is therewith already decided which (infinite) sets of objects the sign represents. \dots
The set of objects represented by $n$ is always exactly the same as before although
the objects themselves represented by $n^2$ are not quite the same as those
represented by $n$.

\end{quote}
The problem is what is the exact meaning of concepts designated by $n, n^2, n^4, n^8, \dots $. In WL, it seems that these are sets of numbers obtained by choosing from all numbers. 
In the letter to Zimmermann, Bolzano asserts that all these sets are the same although the objects are not the same. We have returned to our fundamental question. How do we obtain the sets (of squares)? Bolzano's own answer was: it depends on its determination ground. 
\medskip

Jan Berg concluded that Bolzano in the end of his life confined PW and accepted CP (Berg 1973, 27 - 28), but Paul Rusnock (2000, p. 194) and Paolo Mancosu refuse it and find this conclusion unjustified. Mancosu ironically notes
\begin{quote} Thus, Bolzano saved his mathematical soul in extremis and joined the rank of the blessed
Cantorians by repudiating his previous sins. (Mancosu 2009, p. 626) \end{quote}} 

We extend Bolzano's theory of infinite quantities and suggest a method to determine the size of sets that we designate as \emph{calculable}. The size of calculable sets consistently extends the size of finite sets, i.e. its cardinality. PW is valid. We follow two Bolzano principles, which we interpret in the contemporary mathematical context. 
\begin{enumerate}
\item The multitude of terms of infinite series is always \emph{one and the same}. 
\item An equality of two multitudes can be concluded only if some other reason is added such as the same \emph{determining ground} or a \emph{way of being formed}.
\end{enumerate}

\begin{definition}
The set $A$ is called \emph{calculable}, if $A$ can be arranged as a union of disjoint finite sets which can be indexed by natural numbers.\footnote{We denote by $\mathbb N$ natural numbers greater than 0, i.e. $\mathbb N = \{1, 2, 3, \dots \}$. We use this notation rather from aesthetic reasons.} 
$$A = \bigcup\{A_n, n \in \mathbb N\}.$$ 
The arrangement of a calculable set is given by a finite-to-one\footnote{A \emph{finite-to-one} function means that all pre-images $l_A^{-1}(n)$ are finite. See (Benci, Di Nasso 2019, p. 277).} \emph{labelling function} $l_A: A \longrightarrow \mathbb N$ such that
$$l_A(x) = n \Leftrightarrow x \in A_n.$$
\end{definition}

\begin{remark}
\begin{enumerate}
\item Our methodology is close to that of the numerosity theory (NT). We appreciate NT and are grateful for it and for its solid mathematical foundation. However, our conception is based on Bolzano's ideas and was created nearly independently. We shall compare it in Section \ref{num}.
\item We use the provisional working name \emph{calculable} set. It is basically the same as the \emph{labelled} set in NT. We choose a different term on purpose to point out some different assumptions of our theory, and the resulting slightly different conclusions. The notion of a \emph{labelling function} is also borrowed from NT. 

\item Each calculable set is \emph{countable} by Cantor's definition. But it does not have to work the other way around. This applies in particular to ordinal numbers. They are concepts based on CP, which is in dispute with PW. We will not deal with their sizes. The same ordinal number $\omega$ has the set od natural numbers and all its infinite subsets. Ordinal numbers are types of well-orderings. It is meaningless to determine their sizes. 

\end{enumerate}
\end{remark}

\subsection{Characteristic and size sequences}

Let $A = \bigcup\{A_n, n \in \mathbb N\}$ be a calculable set. From the set-size point of view we can describe $A$ by its \emph{characteristic sequence} $\chi(A)$ such that its n-th term is $|A_n|$, the cardinality of $A_n$. The \emph{size} of $|A|$ is represented by the size sequence $\sigma(A)$ which is the interpretation of the Bolzano series 
$$|A| = |A_1| + |A_2| + |A_3| + \dots \text{in inf.} \sim \sigma(A) = (\sigma_n(A)),$$

\begin{remark} We consider the sequence $\sigma(A)$ a single, exactly given \emph{quantity}. There is a conceptual distinction between treating sequences as classical sequences and treating them as interpretations of Bolzano's series. In the former viewpoint they represent unique and exactly given quantities, similarly as in non-standard analysis.
\end{remark} 

\begin{definition} Let $A = \bigcup\{A_n, n \in \mathbb N\}$, where $A_n$ are finite sets, be a calculable set. 
\begin{itemize} 
\item A \emph{characteristic sequence} of a set $A$ is the sequence $\chi(A) = (\chi_n(A)) \in \mathbb N^\mathbb N$ where $$\chi_n(A) = |A_n|.$$
\item A \emph{size sequence}\footnote{The size sequence corresponds to the \emph{approximating sequence} in NT.} of a set $A$
is the sequence $\sigma(A) = (\sigma_n(A)) \in \mathbb N^\mathbb N$ such that $$\sigma_n(A) = |A_1| + \dots + |A_n|.$$
\end{itemize}
\end{definition}

The notions of arrangement, labelling function, characteristic and size sequence are closely connected. One follows from the knowledge of the other. We introduced them all here for better comprehension. Simple relationships apply between them.

\begin{proposition} Let $(\chi_n(A))$ be a characteristic sequence, $(\sigma_n(A))$ a size sequence and $l_A(x)$ a labelling function of a calculable set $A$. Then for all natural numbers $n \in \mathbb N$

\begin{itemize} 
\item $A_n = \{x \in A; l_A(x) = n\}.$
\item $\sigma_n(A) = \sigma_{n-1}(A) + \chi_n(A)$
\item $\chi_n(A) = |\{x; l_A(x) = n\}|. $
\item $\sigma_n(A) = |\{x; l_A(x) \leq n\}| .$

\end{itemize}
\end{proposition}
\subsection{Equality and order}

Size sequences represent sizes of calculable sets. We define their arithmetic operations componentwise according to Definition \ref{plus}. We use the Fr\'{e}chet filter $\mathcal F$ to define an \emph{equality} $=_\mathcal F$ and an \emph{ordering} $<_\mathcal F$ as in Definition \ref{equality}. Similarly, we define a relation \emph{less or equal} $\leq_\mathcal F$. 
\begin{definition}\label{Frechet}
Let $A, B$ be two calculable sets. Their size sequences are $\sigma (A) = (\sigma_n(A))$, $\sigma (B) = (\sigma_n(B))$. We define
\begin{itemize}
\item $\sigma (A) =_\mathcal F \sigma (B), \text{ if and only if } (\exists m)(\forall n)(n > m \Rightarrow \sigma_n(A) = \sigma_n(B))$.
\item $\sigma (A) <_\mathcal F \sigma (B), \text{ if and only if } (\exists m)(\forall n)(n > m \Rightarrow \sigma_n(A) < \sigma_n(B))$.
\item $\sigma (A) \leq_\mathcal F \sigma (B), \text{ if and only if } (\exists m)(\forall n)(n > m \Rightarrow \sigma_n(A) \leq \sigma_n(B))$.
\item We identify the constant sequence $(k) \in \mathbb N^\mathbb N$ with the number $k \in \mathbb N$ as usually. 
\end{itemize}

\end{definition} 

\begin{proposition}\label{finite} Let $A, B$ be calculable sets.
Let $S \subseteq \mathbb R^\mathbb N$ be the set of non-decreasing sequences of natural numbers.
\begin{enumerate}
\item If $\sigma ( A) <_\mathcal F \sigma ( B)$ or $\sigma (A) =_\mathcal F \sigma (B)$ then $\sigma (A) \leq_\mathcal F \sigma ( B)$. The reverse implication need not hold true.
\item For any set arrangement of a set $A$ 
$$|A| = n \text{ if and only if } \sigma(A) =_\mathcal F (n).$$

\end{enumerate}

\end{proposition}

\begin{proof}
\begin{enumerate}
\item Directly from Definition \ref{Frechet}.
\item The arrangement of a finite set is unimportant. The terms of its size sequence are equal to $n$ from a sufficiently large index. 
\end{enumerate}
\end{proof}

However, this is not the case with infinite sets, see the following example. 

\subsection {Canonical arrangement}
\begin{example}
The set of natural numbers $\mathbb N$ is defined as integers greater than zero
$$\mathbb N = \{1, 2, 3, 4, \dots \}$$
This is a calculable set. Let us consider three of its possible arrangements: 
\begin{enumerate}[1.]

\item Let $l_A(n) = n$ for all $n$. It means $A_n = \{n\}$ for all $n$. Then $\mathbb N = \bigcup\{A_n, n \in \mathbb N\}$, $$\chi_A(\mathbb N) = (1,1,1,1,\dots), \quad \sigma_A(\mathbb N) = (1,2,3,4, \dots)$$ 
\item Let $l_B(n) = n+1$ for odd $n$, $l(n) = n$ for even $n$. It means $B_n = \emptyset$ for odd $n$, $B_n = \{n-1, n\}$ for even $n$, Then $\mathbb N =\bigcup\{B_n, n \in \mathbb N\}$, 
$$\chi_B(\mathbb N) = (0,2,0,2,\dots), \quad \sigma_B(\mathbb N) = (0,2,2,4, \dots).$$ 
\item Let $l_C(n) = n$ for odd $n$, $l(n) = n-1$ for even $n$. It means $C_n = \{n, n+1\}$ for odd $n$, $C_n = \emptyset$ for even $n$. Then $\mathbb N = \bigcup\{C_n, n \in \mathbb N\}$, $$\chi_C(\mathbb N) = (2,0,2,0, \dots), \quad \sigma_C(\mathbb N) = (2,2,4,4, \dots).$$ 
\end{enumerate}
The size sequences are different, though their difference is less or equal to 1. $$\sigma_B(\mathbb N) \leq_\mathcal F \sigma_A(\mathbb N) \leq_\mathcal F \sigma_C(\mathbb N) \leq_\mathcal F \sigma_B(\mathbb N) + 1.$$ 
\end{example}

We introduce a \emph{canonical labelling} and the following \emph{canonical arrangement} of calculable sets, 
which corresponds to the \emph{determining ground} or to the \emph{way of beimg formed}. We define a canonical labelling sequentially from the simple to the more complex sets so that it is unambiguous. 

\begin{definition} 
A calculable set $A$ has a \emph{canonical labelling} $l_A(x): A \rightarrow \mathbb N$ if the following conditions hold: 
\begin{enumerate}

\item 
If $B$ is a canonically arranged set and $x \in A \cap B$ then $$l_A(x) = l(x).$$ 
\item The canonical arrangement corresponds to the determining ground.
\end{enumerate} 
If $A$ has a canonical labelling then $A$ is \emph{canonically arranged}.
\end{definition}
The second condition is not formulated precisely. Its meaning is sometimes intuitively clear, but sometimes we have to look for it carefully.

\begin{theorem}\label{sigma}
If $A, B$ are two calculable canonically arranged sets then 
$$\sigma(A \cup B) = \sigma(A) + \sigma(B) - \sigma(A \cap B)$$
\end{theorem}

\begin{proof}
It follows from the uniqueness of the canonical arrangement. 
\end{proof}

\begin{consequence} Let $A, B$ be two calculable sets. Then
\begin{enumerate}[(i)]
\item $\sigma(A) = 0$ if and only if $A = \emptyset$.
\item If $A$ is finite, i.e. $(\exists n \in \mathbb N)(|A| = n)$, then $\sigma(A) =_\mathcal F (n)$.
\item If $A$ is a proper subset of $B$, i.e. $A \subset B$, then $\sigma(A) < \sigma(B)$
\footnote{The Part-whole principle.}
\item If $A, B$ are disjoint, i.e. $A \cap B = \emptyset$, then $\sigma (A \cup B) = \sigma ( A) + \sigma ( B)$.
\item $\sigma(A) < \sigma(B) \Rightarrow \sigma(A) + 1 \leq \sigma(B)$\footnote{This property is called a \emph{discreteness}, (Parker 2013).}
\end{enumerate}
\end{consequence}

\subsection{Natural numbers} 

We start with the canonical labelling of natural numbers, which is evident, $l(n) = n$. Each set with the same arrangement as natural numbers has the same labelling function. 

\begin{definition}
\begin{itemize} Let $\mathbb N = \{1, 2, 3, 4, \dots \}$ be the set of all natural numbers.
\item The \emph{canonical labelling} of natural numbers $\mathbb N$ is the function $l: \mathbb N \longrightarrow \mathbb N$ such that $$l(n) = n.$$
\item If $B$ is a set which is indexed by natural numbers, i. e. $B = \{b_n, n \in \mathbb N\}$, then its canonical labelling is the function $l: B \longrightarrow \mathbb N$ such that
$$l(b_n) = n.$$ 
\end{itemize} 
\end{definition}

\begin{example} Let $\mathbb N$ be the set of natural numbers. 
\begin{enumerate}
\item $\chi(\mathbb N) = (1,1,1,1, \dots)$, 
$\sigma(\mathbb N) = (1,2,3,4, \dots).$
\item $\chi( \{3, 4\}) = (0,0,1,1,0,0 \dots)$, 
$\sigma( \{3, 4\}) = (0,0,1,2,2,2, \dots) =_\mathcal F 2$. 
\item $\chi(\mathbb N \setminus \{3, 4\}) = (1, 0, 0, 1, 1, 1, \dots)$, 
$\sigma(\mathbb N \setminus \{3, 4\}) = (1, 2, 2, 2, 3, 4, \dots) =_\mathcal F \sigma(\mathbb N) - 2$.
\item Let $E$ be the set of even numbers. Then $\chi(E) = (0,1,0,1,0,1, \dots)$, 
$\sigma(E) = (0,1,1,2,2,3,3 \dots)$. 
\item Let $O$ be the set of odd numbers. Then $\chi(O) = (1,0,1,0,1,0 \dots)$, 
$\sigma (O) = (1,1,2,2,3,3, \dots)$. 
\item Let $S$ be the set of square numbers, $S = \{1, 4, 9, 16, \dots\}$. Then $\chi(S) = (1,0,0,1,0,0,0,0,1,0 \dots)$, 
$\sigma (S) = (1,1,1,2,2,2,2,2,3,3 \dots) <_\mathcal F \sigma(\mathbb N)$ .

\end{enumerate}
\end{example}

\begin{remark} Size sequences of subsets of natural numbers are non-decreasing sequences such that the following term is always either equal to the preceding one or to the preceding term plus 1, i.e. if
$A \subseteq \mathbb N$ and $\sigma(A) = (a_n)$ is its size sequence then 
$$(\forall n)(a _{n+1} = a_n \vee a _{n+1} = a_n+1) $$ 
And vice versa, if $(a_n)$ is a non-decreasing sequence of natural numbers with this property then there is a subset $A$ of natural numbers such that $(a_n)$ is its size sequence. 
\end{remark}

\begin{definition}
The sequence $\sigma (\mathbb N)$ represents the size of natural numbers. We denote it symbolically as
$$\alpha = \sigma (\mathbb N)= (1, 2, 3, \dots)$$ 
\end{definition}

The sequence $\alpha$ is the interpretation of Bolzano's series 
$$\alpha \sim 1 + 1 + 1 + \dots \text{in inf.} = N_0$$
which represents the multitude of all natural number. It is neither the greatest finite number nor the smallest infinite number, see the next Section \ref{omega}.

\subsection{What is $\alpha$? }\label{omega}

There is a significant difference between Cantor's $\omega$ and the sequence $\alpha$ which is the interpretation of Bolzano's series
$$N_0 = 1+ 1+ 1+ \dots \text{in inf.} \sim \alpha = (1,2,3, \dots )$$ 
The series $N_0$, as well as $\alpha$ is a fixed quantity. It is infinite, i.e. it is greater than every finite number. It is on \enquote{the same level} as the terms it contains, see Section \ref{collection}. It is neither the greatest finite number 
\begin{quote}
\dots because it is a self-contradictory concept (PU \S 15)\end{quote}
nor the smallest infinite number. If we subtract from $\alpha$ any finite number $m \in \mathbb N$ then still the difference $\alpha - m$ 
remains infinite. 

We do not have to require even natural $\mathbb N$ numbers to form a set. They can be a proper class. \footnote{This conception is built and philosophically justified in the \emph{Alternative Set Theory} of Petr Vop\v{e}nka (Vopenka 1974). Natural numbers do not form a set but a \enquote{semiset} which is a vague part of some greater set. The \emph{Prolongation axiom} guarantees the existence of an infinitely great number. It is motivated phenomenologically.} 
We only assume there is an infinite number $\alpha$ that is greater than all finite numbers. It is similarly in a non-standard model of natural numbers or in the numerosity theory, where $\alpha$ is defined as a \enquote{new} number which can be considered as the sequence $(n)$. (Benci, Di Nasso 2019, p. 15).

\subsection{Partial or linear order?}\label{partial order}
\begin{example}Let $E$ be the set of even numbers, $O$ the set of odd numbers. 
\begin{enumerate}

\item $\sigma(E) <_\mathcal F \alpha$, $\sigma(O) <_\mathcal F \alpha.$
\item $\sigma(E) + \sigma (O) = \alpha$
\item \label{even} $\sigma(E) \leq_\mathcal F \sigma (O) \leq_\mathcal F \sigma(E) + 1$\footnote{If we defined natural numbers as the set $\{0, 1, 2, \dots \}$ then the relation among the sizes of even numbers and odd numbers would be opposite.} 
\item $0 \leq_\mathcal F \sigma (O) - \sigma (E) \leq_\mathcal F 1.$ 
\end{enumerate}
\end{example}

The relationship between sizes of some sets is not entirely determinable. Neither $\sigma(E) <_\mathcal F \sigma (O)$ nor $\sigma(O) <_\mathcal F \sigma (E)$. However their difference is less than or equal to 1.

We see that the ordering of size sequences is only partial and not linear.\footnote{The ordering $<$ defined on a set $S$ is \emph{linear (total)} if for every $x,y \in S$ one of the three options holds true $$x < y \vee x = y \vee y < x.$$} If we used a non-principal ultrafilter in the Definition \ref{Frechet} instead of the Fr\'{e}chet filter we could easily extend it to a linear ordering. But the result would depend on the choice of an ultrafilter which is to a certain extent accidental (Benci, Di Nasso 2003).

\begin{example} If we used an ultrafilter $\mathcal U$ for the determination of the relation of the sizes of odd and even numbers then there are two options: 
\begin{itemize}
\item If $E \in \mathcal U$ then $\sigma(O) =_\mathcal U \sigma(E)$, thus $\alpha = 2 \cdot \sigma(E)$ and $\alpha$ is even.
\item If $E \notin \mathcal U$ then $\sigma(O) =_\mathcal U \sigma(E) + 1$, thus $\alpha = 2 \cdot \sigma(E) + 1$ and $\alpha$ is odd. 
\end{itemize}

The answer to the question of whether $\alpha$ is even or odd depends on the choice of an ultrafilter.
\end{example}

We use only Fr\'{e}chet filter and the results of relations $=_\mathcal F, <_\mathcal F, \leq_\mathcal F$ are given uniquely. Moreover, Fr\'{e}chet filter is contained in every ultrafilter. Therefore, these relations are valid for any eventual extension to an ultrafilter. For these reasons, it seems more appropriate to keep using Fr\'{e}chet filter and admit that we cannot determine some properties of infinite quantities.

It is quite natural that our knowledge of the relationships of the sizes of some infinite sets is not complete. Already Bolzano wrote: 

\begin{quote}Even the concept of a calculation of the infinite has, I admit, the appearance of being self-contradictory. \dots But this doubtfulness disappears if we take into account that 
a calculation of the infinite done correctly does not aim at a 
calculation \dots of infinite plurality in itself, 
but only a determination of the \emph{relationship} of one infinity to another. This is a matter which is feasible, in certain cases at any rate. 
(PU \S 28.) \end{quote}

\subsection{Set sizes of multiples, powers and primes}

As an example, we demonstrate how to determine sizes of some subsets of natural numbers and their relationship.
\begin{theorem}
Let $k \in \mathbb N$. 
\begin{enumerate}
\item The size of all primes $P$ is less than the size of multiples $M_k$ for any $k \in \mathbb N$.
\item The size of all primes $P$ is greater than the size of k-th powers $S_k$ for any $k \in \mathbb N$. 
\end{enumerate} 
\end{theorem}

\begin{proof}
Let $P$ be the set of all primes. The number of primes which are less or equal $n$ is usually denoted $\pi(n)$. The size sequence of primes $\sigma(P) = (\pi(n)) = (1,2,3,3,4,4,5,5,5,5, \dots)$. The well-known fact about primes is that for $n>10$ 
$$\frac{n}{\log n} \leq \pi(n) \leq \frac{3n}{\log n}$$

\begin{enumerate}
\item For any $k \in \mathbb N$ we denote a set of $k$-multiples as $M_k = \{n \in \mathbb N; (\exists m \in \mathbb N)(n = m\cdot k)\}$. Its size sequence is $\sigma(M_k) = (\underbrace{0, \dots 0}_{k-1}, \underbrace{1, \dots 1}_k, \underbrace{2 \dots 2}_k, \dots \dots )$. So
$$\frac{n - 1}{k} \leq \sigma_n (M_k) \leq \frac{n}{k}.$$
Because $$(\forall k)(\exists m)(\forall n > m)(\frac{k}{n - 1} < \frac{\log n}{3n}),$$ there is 
$$(\forall k)(\exists m)(\forall n > m)(\pi(n) \leq \frac{3n}{\log n} < \frac{n - 1}{k} < \sigma_n (M_k)). $$
Consequently
$$\sigma(P) <_\mathcal F \sigma(M_k)$$

\item We denote the set of the second powers of natural numbers as $S = \{n \in \mathbb N; (\exists m \in \mathbb N)(n = m^2)\}$, $S = \{1, 4, 9, 16, \dots \}$. Its size sequence is $\sigma(S) = (1, 1, 1, 2, 2, 2, 2, 2, 3, 3 \dots )$ and it is valid
$$\sqrt{n-1} \leq \sigma_n(S) \leq \sqrt{n}$$ 
Because $$\sqrt{n} < \frac{n}{\log n},$$ there is 
$$\sigma(S) <_\mathcal F \sigma(P).$$

For the set of the k-th powers $S_k = \{n \in \mathbb N; (\exists m \in \mathbb N)(n = m^k)\}= \{1, 2^k, 3^k, \dots\}$ there is valid 
$$\sqrt[k]{n-1} \leq \sigma_n(S_k) \leq \sqrt[k]{n}.$$ 
Because $\sqrt[k]{n} \leq \sqrt{n}$ there is also 
$$\sigma(S_k) <_\mathcal F \sigma(P)$$

\end{enumerate}
\end{proof}

\subsection{Integers}
Integers $\mathbb Z$ can be described $$\mathbb Z = \mathbb N \cup \mathbb N^{-} \cup \{0\}.$$ 
Negative whole numbers $\mathbb N^{-}$ have the same canonical arrangement as natural numbers. 
The arrangement of a finite set is unimportant, see Proposition \ref{finite}, we can define $l(0) = 1.$

\begin{definition} We define the \emph{labelling} of integers $\mathbb Z$ as the function $l_{Z}: \mathbb Z \longrightarrow \mathbb N$ 
$$l_\mathbb Z(x) = |x|\text{ if } x \neq 0 \text{, and } l_\mathbb Z(0) = 1 $$
\end{definition}

\begin{proposition} The labelling function $l_\mathbb Z$ is canonical. For every natural number $n \in \mathbb N \subseteq \mathbb Z$ $$l_\mathbb Z(n) = l(n).$$ 
We can write $l(x)$ instead of $l_\mathbb Z(x)$ for $x \in \mathbb Z$.\end{proposition}

\begin{itemize}
\item The canonical arrangement of $\mathbb Z$: $A_1 = \{0, 1, -1\}, A_n = \{n, -n\}$ for $n \neq 1$.
\item The characteristic sequence $\chi(\mathbb Z) = (3, 2, 2, 2, \dots)$.
\item The size sequence $\sigma(\mathbb Z) = \sigma(\mathbb N) + \sigma(\mathbb N^{-}) + \sigma(\{0\}) = 2 \alpha + 1$. 
\item $\sigma(\mathbb Z) = (1,2,3,4,5, \dots) + (1,2,3,4,5, \dots) + (1,1,1,1,1, \dots) = (2n+1) = (3, 5, 7, 9, 11, \dots)$.
\end{itemize}

\subsection{Cartesian product}

We arrange the Cartesian product of two computable sets $A, B$, $A = \bigcup\{A_n, n \in \mathbb N\}, B = \bigcup\{B_n, n \in \mathbb N\}$ as a union of \enquote{frames}. The $n-$th \enquote{frame} $(A \times B)_n$ contains finitely many elements. 
$$(A \times B)_n = \bigcup\{A_i \times B_j, n = \max\{i,j\}\}.$$ 
\enquote{Frames} are borders of \enquote{squares}, the n-th \enquote{frame} is a border of the n-the \enquote{square}. A characteristic sequence is determined by numbers of elements in \enquote{frames}, a size sequence by numbers of elements in \enquote{squares}. 
\begin{center}
\catcode`\-=12
\begin{tabular}{c||c c c c c c c c c c}
& $A_1$ & $A_2$ & $A_3$ & $A_4$ & $\dots$\\
\hline
\hline
$B_1$ & \ $A_1\times B_1$ \.& \ $A_2\times B_1$ \. & \ $A_3\times B_1$ \. & \ $A_4\times B_1$ \. & \dots \\
\cline{2-2}
$B_2$ & $A_1\times B_2$ & \ $A_2\times B_2$ \.& \ $A_3\times B_2$ \. & \ $A_4\times B_2$ \. & \dots \\
\cline{2-3}
$B_3$ & $A_1\times B_3$ & $A_2\times B_3$ & \ $A_3\times B_3$ \. & \ $A_3\times B_3$ \. & \dots \\
\cline{2-4}
$B_4$ & $A_1\times B_4$ & $A_2\times B_4$ & $A_3\times B_4$ & $A_4\times B_4$ \ \. & \dots\\
\cline{2-5}
$\dots$ &\dots &\dots & \dots &\dots & \dots \\
\end{tabular}
\end{center}
\begin{definition} 
\item Let $A, B$ be two computable sets. The \emph{labelling} of $A \times B$ is the function $l_{A \times B}: A \times B \longrightarrow \mathbb N$ such that for all $[x, y] \in A \times B$ $$l_{A \times B}([x, y]) = \max\{l_A(x),l_B(y)\}$$ 

\end{definition}

\begin{theorem}\label{Cartesian} Let $A, B$ be two calculable sets. 
\begin{enumerate}

\item The Cartesian product $A \times B$ is calculable.
\item If $A, B$ are canonically arranged then their Cartesian product is canonically arranged too. It holds
$$l([x, y]) = \max\{l(x),l(y)\}$$ 
\item If $\sigma(A), \sigma(B)$ are size sequences of $A,B$ then the size sequence of $A \times B$ is $$\sigma(A \times B) = \sigma(A) \cdot \sigma(B).$$
\item If $C \subseteq A \times B$ then its size sequence $\sigma(C)$ 
$$\sigma_n(C) = |\{[x,y] \in C \wedge l([x,y]) \leq n\}|.$$
\end{enumerate}
\end{theorem}

\begin{proof}
\begin{enumerate}
\item $A \times B = \bigcup\{(A \times B)_n , n \in \mathbb N\}$, where $(A \times B)_n = \bigcup\{A_i \times B_j, n = \max\{i,j\}\}$ is a finite set. 
\item The size sequence of a canonically arranged set $A$ is the same as the arrangement of $A \times \{1\}$.
\item We wish to prove $\sigma_n(A \times B) = \sigma_n (A) \cdot \sigma_n (B)$. 

$\sigma_n(A \times B) = |\{[x,y] \in A \times B, \max\{l(x),l(y)\} \leq n\}| =$
$|\{[x,y] \in A \times B, l(x) \leq n \wedge l(y) \leq n\}| =$
\item $|\{x \in A, l(x) \leq n\}| \cdot |\{y \in B, l(y) \leq n\}| = \sigma_n (A) \cdot \sigma_n (B)$
\end{enumerate}

\end{proof}

\begin{example}

Let $\mathbb N \times \mathbb N$ be the set of pairs of natural numbers.

\begin{itemize} 
\item The canonical arrangement of $\mathbb N \times \mathbb N$: $A_1 = \{[1,1]\}, A_2 = \{[1,2], [2,2],[2,1]\}, \dots $
\item $\chi(\mathbb N \times \mathbb N) = (1, 3, 5, 7, 9, 11, 13, 15, 17, 19 \dots) = (2n-1)$
\item $\sigma(\mathbb N \times \mathbb N) = (1,4, 9, 16, 25, 36, 49, 64, 81, 100 \dots) = (n^2) = (n)^2 = \alpha^2.$
\end{itemize}
\end{example}

\begin{example} The size sequence of even numbers $E$ is $\sigma (E) = (0,1,1,2,2,3,3, \dots)$ and of odd numbers $O$ is $\sigma (O) = (1,1,2,2,3,3,4 \dots)$. Thus
\begin{itemize}
\item $\sigma(E \times O) = \sigma (E) \cdot \sigma (O) = (0, 1, 2, 4, 6, 9, 12, 16, 20, \dots)$
\item $\sigma(O \times E) = \sigma (O) \cdot \sigma (E) = (0, 1, 2, 4, 6, 9, 12, 16, 20, \dots)$
\item $\sigma(E \times E) = \sigma (E)^2 = (0, 1, 1, 4, 4, 9, 9, 16, 16, \dots)$
\item $\sigma(O \times O) = \sigma (O)^2 = (1, 1, 4, 4, 9, 9, 16, 16, 25, \dots)$ 
\item $\sigma(\mathbb N \times \mathbb N) = \sigma(E \times O) + \sigma(O \times E) + \sigma(E \times E) + \sigma(O \times O) = (1, 4, 9, 16, 25, 36, 49, 64, 81, \dots)$
\end{itemize}
\end{example}

\subsection{Rational numbers}

A canonical arrangement of positive rational numbers does not seem to be a problem at first sight. Let us start with the half-open unit interval of rationals $$I = (0,1]_\mathbb Q \subseteq \mathbb Q.$$ 
Every $x \in I$ can be expressed as a proper fraction $x = \frac{m}{n}$, i.e. a ratio of two coprime\footnote{Two numbers $m,n \in \mathbb N$ are \emph{coprime} if their only common divisor is $1$, i.e. $(\forall z \in \mathbb N)((\exists u \in \mathbb N)(\exists v \in \mathbb N)(m = z \cdot u \wedge n = z \cdot v) \Rightarrow z = 1$} natural numbers $m, n \in \mathbb N$ such that $m < n$. 
$$I = (0,1]_\mathbb Q \sim \{[m,n] \in \mathbb N \times \mathbb N; m, n \text{ are coprime and } m< n\} \cup [1,1]\}$$
\begin{center}

\catcode`\-=12
\begin{tabular}{c||c c c c c c c c c c}
& 1 & 2 & 3 & 4 & 5 & 6 & 7 & 8 & 9 & $\dots$\\
\hline
\hline
1 & \. 1 & \. 0 & \. 0 & \. 0 & \. 0 & \. 0 & \. 0 \. & \. 0 & \. 0 & \dots \\
\cline{2-2}
2 & 1 & \. 0 & \. 0 & \. 0 & \. 0 & \. 0 & \. 0 & \. 0 & \. 0 & \\
\cline{2-3}
3 & 1 & 1 & \. 0 & \. 0 & \. 0 & \. 0 & \. 0 & \. 0 & \. 0 & \\
\cline{2-4}
4 & 1 & 0 & 1 & 0\. & \.0 & \.0 & \.0 & \.0 & \.0 & \\
\cline{2-5}
5 & 1 & 1 & 1 & 1 &\. 0 &\. 0 & \.0 & \.0 & \.0 & \\
\cline{2-6}
6 & 1 & 0 & 0 & 0 & 1 &\. 0 & \.0 & \.0 & \.0 & \\
\cline{2-7}
7 & 1 & 1 & 1 & 1 & 1 & 1 & \.0 & \.0 & \.0 & \\
\cline{2-8}
8 & 1 & 0 & 1 & 0 & 1 & 0 & 1 & \.0 & \.0 & \\
\cline{2-9}
9 & 1 & 1 & 0 & 1 & 1 & 0 & 1 & 1 & \.0 & \\
\cline{2-10}
$\dots$ & & & & & & & & & & \dots \\
\end{tabular}
\end{center}

\begin{itemize}
\item The labelling function: 
$l_I({m\over n}) = n$, $l_I(1) = 1$. 
\item The arrangement $I_1 = \{1\}, I_2 = \{\frac{1}{2}\}, I_3 = \{\frac{1}{3}, \frac{2}{3}\} \dots $
\item The characteristic sequence $\chi(I) = (1,1,2,2,4,2, 6,4,6, \dots).$
\item The size sequence $\sigma(I) = (1,2,4,6,10,12,18,22,28, \dots) < \frac{\alpha^2 - \alpha}{2}$.
\end{itemize}

This arrangement of $I$ extends the canonical arrangement of integers and it depends on the way of being formed of proper fractions, we can consider it as \emph{canonical}. 

\begin{remark}
The n-th term of the characteristic sequence $\chi_n(I)$ is equal to the number of coprime numbers less than $n$ which is given by the \emph{Euler's totient function} $\varphi$
$$\chi_n(I) = \varphi(n).$$ 
There are several simple methods for computing Euler's function.\footnote{See for instance https://brilliant.org/wiki/eulers-totient-function/.} 
\end{remark}
$$\ast \ast \ast$$

The best way to express suitably positive rational numbers seems to be as mixed fractions\footnote{Positive rational numbers could also be expressed as fractions of natural numbers. Then they would be represented as pairs of coprime natural numbers which is a subset of $\mathbb N \times \mathbb N$. But it does not seem to be appropriate. Then the canonical arrangement would not be uniform and homogeneous. At the same time, the numbers are increasing in \enquote{depth}, most of them are proper fractions between 0 and 1,
and in the \enquote{distance}, where they are very sparse. Consequently, the size of rational numbers between 0 and 1 would be the same as the size of rational numbers greater than 1 and it cannot be true. 
$$\sigma(0,1)_\mathbb Q = \sigma(1, \infty)_\mathbb Q$$.}
Rational numbers are represented as a Cartesian product of $\mathbb N_0$\footnote{We denote by $N_0$ natural numbers including zero, i.e. $N_0 = \{0,1,2,3, \dots \}$.} and a unit interval $I$. 
$$\mathbb Q^+ = \mathbb N_0 \times I \sim \{p + {m \over n}, p \in N_0, m,n \in \mathbb N \wedge m< n \}$$
\begin{itemize} 
\item The labelling function: $l_\mathbb Q^+ (p + {m \over n}) = \max\{p, n\}$. 
\item The arrangement $\mathbb Q^+_1 = \{0,1\}, \mathbb Q^+_2 = \{{1\over 2}, 1{1 \over 2}, 2, 2{1\over 2}\}, \mathbb Q^+_3 = \{{1\over 3}, {2\over 3}, 1{1\over 3}, 1{2 \over 3}, 2{1\over 3}, 2{2\over 3}, 3, 3{1\over 3}, 3{1\over 2}, 3{2\over 3}\}, \dots$.
\item The size sequence 
$\sigma(\mathbb Q^+)= \sigma(\mathbb N_0) \cdot \sigma(I) < (\alpha + 1) \cdot \frac{\alpha^2 - \alpha}{2} < \frac{ \alpha^3 - \alpha}{2}.$ 

$ (2,3,4,5,6,7,8,9,10, \dots) \cdot(1,2,4,6,10,12,18,22,28, \dots) = (2,6,16,30,60,84,144,198, 280, \dots ) $
\item The characteristic sequence $\chi(\mathbb Q^+) = (2, 4, 10,14,30, \dots)$

\end{itemize}

This arrangement of rational numbers is uniform, at least in the sense that there is an equal number of elements in each unit interval. We can regard it as a canonical arrangement.

$$\ast\ast\ast$$

We express all rational numbers $\mathbb Q$ as the union of positive rational numbers, negative rational numbers and zero. $$\mathbb Q = \mathbb Q^+ \cup \mathbb Q^- \cup \{0\}$$

So their size sequence is 
$$\sigma(\mathbb Q) = 2\sigma(\mathbb Q^+) + 1 = (5, 13, 33, 61, 121, 169, 289, \dots) < \alpha^3 - \alpha$$

\subsection{Calculable union of calculable sets}
Similarly, we determine a characteristic sequence and a size sequence of a calculable union of calculable sets. We arrange their elements in segments that cover the entire set. The characteristic sequence indicates the number of elements in each segment, the size sequence the number of elements in each square.
\begin{definition} Let $A$ be the calculable union of calculable sets $A_n$ where $A_n = \bigcup\{A_{n j}, j \in \mathbb N \}$. 
$$A = \bigcup\{A_n, n \in \mathbb N\}$$
We define the \emph{labelling function} $l: A \longrightarrow \mathbb N$ 
$$l(x) = \max\{i, j; x \in A_{ij}\}$$
\end{definition}
\begin{theorem} 
A calculable union of calculable sets is calculable.
\end{theorem}

\begin{proof} 
We define the new arrangement of $A$: 
$$A = \bigcup\{C_n, n \in \mathbb N\},$$
where $C_n = \bigcup \{A_{ij}, n = \max\{i, j\}\}$ are finite sets. Thus $A$ is calculable.
\end{proof}

\begin{example} \emph{Algebraic numbers} are roots of polynomials of n-th degrees for some $n$. They can be expressed as roots of polynomials with integer coefficients. 
$$a_0 + a_1 \cdot x + \dots a_n \cdot x^n.$$ 
There are at most $(2 \alpha + 1)^{n+1}$ polynomials of n-th degree. Every polynomial has maximally $n$ roots. Polynomials of n-degree have maximally 
$n \cdot (2 \alpha + 1)^{n+1}$ roots. Algebraic number are a union of polynomials of n-degree for any n. It is a calculable set. 
\end{example}

\section{Numerosity theory}\label{num}

\subsection{Comparison}

This theory of sizes of computable sets arose independently of numerosity theory NT. Nevertheless, NT was a nice surprise and awakening. Its goal is similar and leads to the same conclusions in many questions. We take the liberty of borrowing some terms to contribute to greater common comprehensibility. Some methods used in NT are also an inspiration for us.
However, we differ from NT in some ways. 
\begin{enumerate}
\item Our concept is based on Bolzano's theory of infinite quantities. The set size is expressed as the interpretation of an infinite Bolzano's series. In NT, basic notions are defined and not justified. 
\item While the size of a finite set does not depend on its arrangement, the arrangement of an infinite set is important for its size. That is why we put emphasis on the canonical arrangement that depends on a determination ground of a set. It is again a Bolzano's idea. 
\item We have shown how and why the size of a Cartesian product of two sets follows from 
the canonical arrangement of their Cartesian product. 
\item The size of rational numbers is examined and justified by their determination ground. However, NT claims:
\begin{quote} In particular, it seems there is no definitive way to decide whether $\mathfrak n_\alpha((0,1]_\mathbb Q) \geq \alpha$ or $\mathfrak n_\alpha((0,1]_\mathbb Q) \leq \alpha$. So, in absence of any reason to choose one of the two possibilities, we go for the simplest option $\mathfrak n_\alpha((0,1]_\mathbb Q) = \alpha$. (BD 2019, p. 291) 
\end{quote}
\item The greatest difference is that we do not use ultrafilters. Set sequences are factorized by Fr\'{e}chet filter. 
Consequently, our ordering of set sizes is only partial and not linear, see Section \ref{partial order}. Nevertheless, our results are unequivocal, while in NT some results are arbitrary. (Benci, Di Nasso, 2003). One can influence it to a certain extent. In (Benci, Di Nasso, 2019, p. 288 - 289), authors postulate that the infinite number $\alpha$ is a multiple of $k$ and it is a k-th power of some number for every $k \in \mathbb N$.

\item NT has one more reason for a need of a linear ordering and so an ultrafilter. Numerosities of countable sets are used as a basis for the $\alpha$-calculus, a special kind of a non-standard analysis. 

\begin{quote} \dots because it is the very idea of numerosity that leads to Alpha-calculus. (Benci, Di Nasso 2019, p. 300) \end{quote}
Bolzano's theory can also be used as a basis for non-standard analysis. However, it is preferable to build on of Bolzano's infinite series of \emph{real numbers} rather than natural numbers. If we factorize them according to the Fr\'{e}chet filter, we get a non-Archimedean ring that is enough to implement the so-called cheap version of non-standard analysis. If we want an elegant full version, in which the transfer principle applies, an ultrafilter has to be used. (Trlifajov\'{a} 2018)

\item Last but not least. Our concept is simpler and more intuitive. All we need is a basic mathematical knowledge. 
Our intuition is based on the idea of how we would proceed if we were to count a large number of things. According to their \emph{determination ground}, we arrange them into smaller groups. The numbers of elements of the groups are written down gradually, it is a \emph{characteristic sequence}, and they are summed, we get \emph{size sequences}. It does not matter how finitely many elements are arranged. This justifies the use of \emph{Fr\'{e}chet filter.}
\end{enumerate}

\subsection{Parker's objections}\label{Parker}

Matthew Parker in his paper \emph{Set Size and the Part-Whole Principle} argues there can be no good theories of set size satisfying the Part-Whole principle, he calls them Euclidean theories. His criticism is mainly focused on the numerosity theory. The first part of the article deals with the general properties of Euclidean theories.
\begin{quote}
But our main question here is whether it is possible to have a really good Euclidean
theory of set size, and my answer is, no, not really - not if that means it must be strong,
general, well-motivated, and informative. I will argue that any Euclidean theory strong
and general enough to determine the sizes of certain simple, countably infinite sets must
incorporate thoroughly arbitrary choices. (Parker 2013, p. 590).
\end{quote}

The concept of size sequences is resistant against all these Parker's objection. It is motivated and based on Bolzano's theory. It extends naturally the notion of a size of finite sets. Set-sizes of simple countably infinite sets are uniquely determined. They are not arbitrary and do not involve unmotivated details. The concept provides an intuitive method for a comparison of set sizes.
\medskip

It is necessary to say that Parker assumes that an assignment determining the set size is a function from a class of objects $D$ to the \emph{linearly} ordered mathematical structure and an effective method to determine the value of a given object. The structure of non-decreasing sequences of natural numbers is not linearly ordered and consequently some set-sizes are not comparable. Nevertheless, Parker admits that sizes of some sets are incomparable. 

\begin{quote} An assignment is Euclidean on $D$ if PW applies to all proper subset/superset pairs in $D$. In that case, all such pairs have size relation, but this does not imply totality; we might for example have disjoint sets in $D$ that are not comparable at all. (Parker 2013, 593).
\end{quote} 
The second part of the paper concerns various types of transformations and rotations. Both of them are one-to-one correspondences which preserve some other property. They usually map an infinite set on its own proper subset. Parker demonstrates, that a Euclidean assignment do not preserve transformations and rotations. It violates a principle that if $T$ is a transformation 
then the size of $A$ should be equal to the size of $T(A)$. 
$$\sigma(A) = \sigma(T(A))$$
Translations are transformations which preserve a distance. For instance, $T(n) = n+1$ is the translation by one unit. (Parker 2013, 595). The translation of $\mathbb N_0 = \{0, 1, 2, \dots\}$ is $T(\mathbb N_0) = \mathbb N = \{1, 2, 3 \dots\}$. So it should be $ \sigma(\mathbb N_0) = \sigma(T(\mathbb N_0))$. 
But at the same time $\mathbb N_0 = \mathbb N \cup \{0\}$ and consequently
$\sigma(\mathbb N_0) = \sigma(\mathbb N) + 1$.
We must protest against this request. It is a well-known fact that infinite sets are NOT invariant to one-to-one correspondence. Already the Dedekind's definition of an infinite set is that it can be put in one-to-one correspondence with a proper subset of itself. We return again to the Galileo's paradox. 
If we want to avoid the contradiction, we must give up the request of a transformation and rotation invariance principles of a size of infinite sets.

On the other side we admit that we have not stated precisely a condition for an extension of a canonical arrangement. We defined a slightly vague condition that it depends on the \emph{determination ground} of a set. If we considered two sets such that one of them did not have a canonical arrangement we would not have an exact criterion for equality of their sizes. Parker's translation, which is the same as Bolzano's isometry of two multitudes (see Section \ref{PW}), is not a sufficient criterion generally. 

\section{Epilogue}

In a letter to Dedekind from July 28, 1899, Georg Cantor, after the first paradoxes had appeared in the set theory, suggested to divide sets into consistent and inconsistent.
\begin{quote} For on the one hand a multiplicity can be such that the assumption that \emph{all} its elements \enquote{are together} leads to a contradiction, so that it is impossible to conceive of the multiplicity as a unity, as \enquote{one finished thing}. Such multiplicities I call \emph{absolutely infinite} or \emph{inconsistent multiplicities}. 
When on the other hand the totality of elements of a multiplicity can be thought without contradiction as \enquote{being together}, so that their collection into \enquote{one thing} is possible I call it a \emph{consistent multiplicity} or a \emph{set}.
(Halett 1984, p. 166).
\end{quote}
Cantor presents as inconsistent multiplicities a totality of everything thinkable, a class of all ordinal numbers $\Omega$ or of all cardinal numbers. Now, they are called proper classes in ZF. In the following letter to Dedekind from August 28, 1899, Cantor returns to the question and asks how do we know that the set of natural number represented by $\aleph_0$ and other cardinal numbers are actually \enquote{consistent multiplicities}. 

\begin{quote}
Is it not conceivable that these multiplicities are already \enquote{inconsistent} but that the contradiction which results from the assumption that \enquote{all their elements can be taken together} has not been yet noticed? \end{quote}
Cantor admits he has no other reason than the similarity with finite sets, whose consistency is a simple, undemonstrable truth. 
\begin{quote}
And likewise the \enquote{consistency} of multiplicities to which I assign the alephs as cardinal numbers is \enquote{the axiom of the extended or the transfinite arithmetic}. (Halett 1984, p. 175) \end{quote}

The existence of a \emph{set} of natural numbers is guaranteed by the axiom of infinity in ZF or its analogy in other set-theoretic systems. But it is more our belief than the proof that the collection of natural numbers is really \enquote{one finished thing}. 
$$\ast\ast\ast$$

Cardinal numbers, ordinal numbers and size sequences express different views on a size and an ordering of infinite sets. For finite sets, these terms match each other because they all correspond to number of elements. They differ for infinite sets. However, they are not mutually exclusive. They would have a good meaning, even if we had doubts that the natural numbers were \enquote{consistent multiplicities}.

\emph{Cardinal numbers} or \emph{powers} denote sizes of sets neglecting one-to-one correspondence. It is typical for infinite sets they are in one-to-one correspondence with some of their proper subsets and also with some of their proper supersets. Consequently, PW cannot hold. 

\emph{Ordinal numbers} represent types of well-orderings. They designate a way of ordering of a set up to isomorphism. It is well-known that every infinite set can be well-ordered by infinitely many ways which can be described by infinitely many ordinal numbers. Of course, PW cannot hold too. 

The notion of \emph{size sequences} may be the closest to the notion of a number of elements of finite sets. Sizes of sets are firmly given, denoted by number sequences. The part-whole principle is valid. The method is constructive and the results are unequivocal. We do not use an ultrafilter, consequently the ordering of set sizes is only partial and not linear. 

We should not suppose that natural numbers create a set, it suffices to suppose that we can consistently work with infinite quantities, quantities which are greater than all finite numbers. That is what Bolzano wanted to demonstrate in his \emph{Paradoxes of the Infinite}.

\end{document}